\documentclass[12pt]{article}

\usepackage[utf8]{inputenc}
\usepackage[english]{babel}
\usepackage{amsmath,amssymb,latexsym,amsfonts}
\usepackage{amsthm}  
\newtheorem{theorem}{Theorem}  
\newtheorem{lemma}{Lemma}
\newtheorem{corollary}{Corollary} 
\newtheorem{exmpl}{Example}
\newtheorem{definition}{Definition}
\newtheorem{remark}{Remark}
\usepackage[cal=boondox]{mathalfa}

\begin{document} 
	\begin{center}
    \LARGE \textbf{On Lacunarity and Sidon-Type Properties for Subsystems of Characters of Compact Abelian Groups}
\end{center}

\begin{center} \Large
Anna Kazakova \footnote{Lomonosov Moscow State University, anna.kazakova@math.msu.ru}
\end{center}


\begin{center} \textbf{Abstract}
\end{center}
\begin{center}
This paper studies subsystems that generalize Rademacher chaos to arbitrary systems of character of compact abelian groups. The properties of $q$-lacunarity and $\frac{2d}{d+1}$-Sidon are proved for polynomial $d$-chaos constructed from $d$-dissociated subsystems of characters of compact abelian groups. Subsystems of the polynomial chaos type for unbounded Vilenkin systems are also considered and it is shown that they are not systems of $q$-lacunarity for all $q>2$ and $\alpha$-Sidon for all $\alpha\in [1, 2)$.
\end{center}

\section*{Introduction}
	This paper is devoted to the study of lacunarity for series with respect to orthogonal systems of functions. The investigation of $q$-lacunary systems and related problems constitutes a fundamental and actively developing research direction, as presented in the works of A. N. Kolmogorov, S. Banach, J. Marcinkiewicz, W. Rudin, S. Sidon, V. F. Gaposhkin, S. B. Stechkin, J. Bourgain,  J. Jakubowski, S. Kwapien, A.~Pełczyński, R.~E.~Edwards, E.~Hewitt, K.~A.~Ross, A. Bonami, G. Pisier, R.~Blei, B. S. Kashin, S. V. Astashkin,  and others (see, e.g., the monographs \cite{Kashin-Saakyan-1999}--\cite{Blei} and survey article \cite{Gaposhkin}). 
	
	For some recent results in this direction, see \cite{Ast-2024}--\cite{Kaz26}.
	
	In a broad sense, the lacunarity of a system of functions implies that its properties are close in the probabilistic sense to those of systems of independent functions. A system $(\varphi_k)_{k \in \mathbb{N}}$ is said to be $q$-lacunary if it satisfies $L_2$–$L_q$ Khinchine inequality:
	\begin{equation} \label{28042039}
		\left\|  
		\sum\limits_{k=1}^n c_k \varphi_k  
		\right\|_{L_q}
		\le 
		\kappa  
		\big\|  
		( c_k )_{k=1}^n  
		\big\|_{l_2}, 
		\quad 
		\text{$\kappa = \kappa(q) > 0$ does not depend on $c_k \in \mathbb{C}$}.
	\end{equation} 
	A. Ya. Khinchine \cite{Xinch} proved this inequality for one of the classical function systems --- the Rademacher system, which is a subsystem of the Walsh system and consists of functions that are independent in the probabilistic sense. Subsequently, in a number of works (see, e.g., \cite{Kolm-Xinch}, \cite{Marcinkiewicz-Zygmund}, \cite{Karlin}), questions related to Khinchine inequality for systems of independent functions were studied. In \cite{Kazakova-2}, criteria were investigated under which a given subsystem of the Walsh or Vilenkin--Crestenson system constitutes a system of independent functions. Since these subsystems satisfy the conditions for the applicability of the results by J. Marcinkiewicz and A. Zygmund, Khinchine inequality also holds for them.
	
	Closely related to Khinchine inequality is Sidon inequality. A uniformly bounded system \((\varphi_k)_{k=0}^{\infty}\) is called a Sidon system if, for every polynomial with respect to this system, the following estimate for the $L_\infty$-norm holds:
	\begin{equation} \label{Sidon}
		\big\|  
		( c_k )_{k=1}^n  
		\big\|_{l_1} \le  C \left\|  
		\sum\limits_{k=1}^n c_k \varphi_k  
		\right\|_{L_\infty}, \quad 
		\text{$C > 0$ does not depend on $c_k \in \mathbb{C}$}.
	\end{equation}
	S. Sidon \cite{SidonTrig, Sidon} proved that Hadamard lacunary subsystems of the trigonometric or Walsh systems are Sidon systems. Inequality \eqref{Sidon} is called Sidon inequality. Sidon's proof relies on the use of Riesz products, which have become an important tool in the theory of orthogonal series (see, e.g., \cite[Ch. 9, \S 4]{Kashin-Saakyan-1999} and \cite[Ch. VII]{Blei}). A classical example of a Sidon system is the Rademacher system \cite{Astashkin-2017}, which forms a Hadamard lacunary subsystem of the Walsh system.
	
	The behavior in the space $L_{\infty}$ is entirely different for the system consisting of $d$-fold products of distinct Rademacher functions (the Rademacher $d$-chaos): this system is no longer a Sidon system, but nevertheless it is a $\frac{2d}{d+1}$-Sidon system, according to Blei's result \cite[Chap. VII]{Blei}. The latter means that inequality \eqref{Sidon} holds with the $l_{\frac{2d}{d+1}}$-norm on the left-hand side instead of the $l_1$-norm, and the constant $\frac{2d}{d+1}$ cannot be improved. For the case $d=2$, J. E. Littlewood had already shown that the $2$-chaos of Rademacher satisfies an inequality of the type \eqref{Sidon} with the $l_{4/3}$-norm instead of the $l_1$-norm, which is also commonly referred to as Littlewood's $4/3$ inequality. (see \cite[Sec. 6]{Astashkin-2017}).
	
	Although the $d$-chaos is not a Sidon system, according to a result by A. Bonami \cite{Bonami}, \cite{Bonami2}, it is a $q$-lacunary system for all $q \ge 2$. Note that Sidon inequality always implies Khinchine inequality for systems of characters of compact abelian groups, as established by G. Pisier \cite{Pisier} (for the trigonometric system, this was proved earlier by W. Rudin \cite{Rudin-1960}). 
	
The study of $q$-lacunarity and the Sidon property for arbitrary systems of characters of compact abelian groups was initiated in \cite{EHR} and \cite{EHR2}. In particular, it was shown in \cite{EHR} that every system of characters of a compact abelian group contains a subsystem that is $q$-lacunary for all $q > 2$ yet fails to be a Sidon system, thereby answering a question posed by W. Rudin \cite{Rudin-1960}. As noted above, Rademacher $d$-chaoses (for $d \ge 2$) exhibit a similar property. It should be noted that, according to a conjecture of R. Blei (see \cite{Blei}, Ch.~XIII, \S~8), any subsystem of characters of a compact abelian group that is $q$-lacunary for all $q > 2$ must necessarily possess the $\alpha$-Sidon property for some $\alpha \in [1, 2)$, even if it is not a Sidon system. However, this remains an open problem even for subsets of the Walsh system.
	
We consider natural generalizations of the notion of Rademacher chaos to a system of functions $(\gamma_k)_{k\in\mathbb{N}}$ that forms a subsystem of characters of a compact abelian group:
	
	\begin{equation} \label{Aphid}
		A_{\gamma}^d :=\{\gamma_{k_0} (x) \cdot \ldots \cdot \gamma_{k_{d-1}}(x) 
		\text{, where $k_0< ... < k_{d-1}$}\}
	\end{equation}
	is called the tetrahedral polynomial chaos of degree $d$ with respect to the system $(\gamma_k)_{k\in\mathbb{N}}$, the system 
	\begin{equation} \label{Bphid}
		B_{\gamma}^d :=\{\gamma_{k_0} (x) \cdot \ldots \cdot \gamma_{k_{d-1}}(x) 
		\text{, where $k_0\le ... \le k_{d-1}$}\}
	\end{equation}
	is called the polynomial chaos of degree $d$ with respect to the system $(\gamma_k)_{k\in\mathbb{N}}$. Furthermore, as the broadest generalization, we introduce another subsystem of polynomial chaos type: 
	\begin{equation} \label{Cphid}
		C_{\gamma}^d :=\{\gamma_{k_0}^{j_0} (x) \cdot \ldots \cdot \gamma_{k_{d-1}}^{j_{d-1}}(x) 
		\text{, where $k_0 < ... < k_{d-1}$ and $1 \le j_i < m_{k_i}$}\},
	\end{equation}
	\begin{equation} \label{m_k}
		\text{where} \; m_k = \inf\{m\in \mathbb{N} \cup \{+\infty\} \colon \gamma_k^m =1\}.
	\end{equation}	
	For the following Theorems A and B, see \cite{Blei}, Chapter 7, \S 12.
	
	{\bf Theorem A.} { \it The tetrahedral chaos $A_{\gamma}^{d}$ for $\{\gamma_{i}\}_{i=1}^{\infty} \subset \Gamma$ is a dissociated subsystem of characters of a compact abelian group, is a system of q-lacunarity.}
	
	{\bf Theorem B.} { \it The tetrahedral chaos $A_{\gamma}^{d}$ for $\{\gamma_{i}\}_{i=1}^{\infty} \subset \Gamma$ is a dissociated subsystem of characters of a compact abelian group, is a $\frac{2d}{d+1}$-Sidon system.}

    In Theorems A and B, a dissociative system is understood as a set in the sense of Definition (37.12) of \cite{HewRoss}, which is equivalent to 2-dissociativity in the sense of Definition~\ref{dissociate}.
    
	For the following Theorems C, see \cite{KazPlo25}.
	
	{\bf Theorem C.} { \it Let $p \ge 2$. Then polynomial chaos type $C_{R}^d$  with respect to the system $(R_k)_{k\in\mathbb{N}}$, where $R_k$ denote the generalized Rademacher functions with base $p$ (see definition \ref{dfn4}), is a system of q-lacunarity.}
	
	For the following Theorems D, see \cite{Kaz26}.
	
	{\bf Theorem D.} { \it Let $p \ge 2$.  Then polynomial chaos type $C_{R}^d$  with respect to the system $(R_k)_{k\in\mathbb{N}}$, where $R_i$ denote the generalized Rademacher functions with base $p$ (see definition \ref{dfn4}), is a $\frac{2d}{d+1}$-Sidon system.}
	
	As can be seen, Theorems C and D deal with chaoses involving $d$-fold products of generalized Rademacher functions allowing arbitrary repetitions of factors, whereas Theorems A and B consider more general systems but are restricted to tetrahedral chaoses, i.e., those that do not permit index repetition. This naturally raises the question of extending Theorems A and B to polynomial chaoses, which allow for repeated indices. In the present paper, we make progress on this problem.
	
In \ref{section-3} of this paper (theorems \ref{THEOREM-1} and \ref{THEOREM-2}), we transfer the results of theorems $A$ and $B$ to the case of polynomial chaos $B^d_{\gamma}$ constructed by $d$-dissociated subsystems of the characters of compact Abelian groups (see the definition of \ref{dissociate}). We prove Theorems \ref{THEOREM-1} and \ref{THEOREM-2} by employing the constructions of Riesz product-type measures developed in Section \ref{section-2}, together with the estimate from Theorem \ref{th-0} relating the $L_q$-norms of polynomials of a specific structure with respect to the system $B^d_{\gamma}$ to the $L_q$-norm of appropriately chosen polynomials with respect to the system $C_{R}^d$. Note that the proof technique based on comparing arbitrary subsystems of character systems with ``model'' systems, such as the classical Rademacher system and the Rademacher $d$-chaos, is well known (see \cite[Ch. 7]{Astashkin-2017}, \cite[Ch. VII]{Blei}, \cite{Pel}--\cite{jakubowski1979}). This approach enables one to derive results on the behavior of polynomials over special subsystems of character systems in function spaces, provided that the corresponding results are known in the ``model'' case. For instance, \cite{Pel} proves that if the $L_{\infty}$-norms of polynomials with respect to character systems of compact abelian groups are equivalent, then all $L_{q}$-norms ($q \ge 1$) of these polynomials are equivalent as well. In particular, this implies the aforementioned result by J. Pisier \cite{Pisier} on the $q$-lacunarity of Sidon sets of characters of compact abelian groups, obtained when the classical Rademacher system is taken as one of the systems in A. Pełczyński's theorem. In Section \ref{section-3}, we adopt $p$-ary analogs of Rademacher chaos as our ``model'' systems.

In \ref{section-5} (Theorems~\ref{theorem-not-a-sid} and~\ref{theorem-not-q-lak}), we show that if, in Theorems~C and~D, the Vilenkin systems with constant $p$ are replaced by unbounded Vilenkin systems  (i.e. $\sup_{k} p_k = \infty$), then the conclusions of Theorems~C and~D no longer hold. Moreover, the resulting polynomial chaos-type subsystem $C^d_R$ fails to be an $\alpha$-Sidon system for any $1 \le \alpha < 2$.

In \ref{section-6}, we establish a necessary condition for $\alpha$-sidonicity (Theorem~\ref{theorem-a}). Note that this theorem yields an alternative proof (compared to \cite[Ch.~VII, \S~11]{Blei}) of the failure of the $\alpha$-Sidon property for $\alpha < \frac{2d}{d+1}$ for $C^d_R$-type subsystems of Vilenkin systems with constant $p$, and it also plays a crucial role in the proof of Theorem~\ref{theorem-not-a-sid}.

	In \ref{sec-def}, we introduce the basic definitions and notation. In \ref{section-1}, we define $d$-dissociated systems and provide examples. In \ref{section-4}, as a consequence of our results and known theorems on the discretization of integral norms, we establish conditions on the number of points required to discretize the $L_q$-norm in our case.

\section{Definitions and Notation} \label{sec-def}
We use the symbol $:=$ to denote equality by definition. 

$\mathbb{C}$ denotes the set of complex numbers, $\mathbb{N}$ the set of natural numbers, and $\mathbb{N}_0 := \mathbb{N} \cup \{ 0 \}$. 

Let $G$ be a compact abelian group, and let $\mu$ be the normalized Haar measure on $G$ (so that $\mu(G) = 1$).

$L_q(G) := \left\{ 
f \colon G \to \mathbb{C} \colon 
\|f\|_q := \left( \int_G |f(x)|^q \, d\mu(x) \right)^{1/q} < \infty 
\right\}, \quad 1 \le q < \infty.$

$L_\infty(G) := \left\{ 
f \colon G \to \mathbb{C} \colon 
\|f\|_\infty < \infty 
\right\},$

where $\|f\|_\infty := \operatorname{ess\,sup}_{x \in G} |f(x)| 
= \inf \left\{ M \ge 0 : |f(x)| \le M \ \text{for $\mu$-a.e. } x \in G \right\}$.

$\left\| ( c_k )_{k=1}^n \right\|_{\ell_2} 
:= \left( \sum_{k=1}^n | c_k |^2 \right)^{1/2}, \quad c_k \in \mathbb{C}$. 

The cardinality of a set $A$ is denoted by $|A|$.

\begin{definition}
	A \textit{character} $\chi$ of the group $G$ is a continuous mapping from $G$ to the unit circle $\{z \in \mathbb{C} : |z|=1\}$ such that $\chi(x \dot{+} y) = \chi(x)\chi(y)$, where $\dot{+}$ denotes the group operation in $G$.
\end{definition}

\begin{definition} \label{dfn4} 
	Let $\mathcal{P} = (p_k)_{k \in \mathbb{N}_0}$ be a sequence of natural numbers such that $p_k \ge 2$. Consider the group $\mathbb{G}_{\mathcal{P}} := \prod_{k=0}^{\infty} \mathbb{Z}_{p_k}$. Elements of $\mathbb{G}_{\mathcal{P}}$ are conveniently represented as infinite sequences 
	\begin{equation}
		\label{Eq:PG-El}  g = (g_0, g_1, \dots), \quad \text{where } g_k \in \{0, \ldots, p_k-1\}.
	\end{equation}  
	
	The character system of $\mathbb{G}_{\mathcal{P}}$ is known as the Vilenkin system, defined as follows.
	
	Let $\omega_{p_k} := \exp\left(\frac{2\pi i}{p_k}\right)$. The \textit{generalized Rademacher functions} on $\mathbb{G}_{\mathcal{P}}$ are given by
	\begin{equation*}
		R_k(g) = \omega_{p_k}^{g_k}, \quad k \in \mathbb{N}_0,
	\end{equation*}
	where $g \in \mathbb{G}_{\mathcal{P}}$ is of the form \eqref{Eq:PG-El}.
	
	Set $P_0 = 1$ and $P_k = \prod_{j=0}^{k-1} p_j$ for $k \ge 1$. The Vilenkin functions in the Paley ordering are defined as products of generalized Rademacher functions:
	\begin{equation}
		V_0(g) \equiv 1, \qquad V_n(g) = \prod_{k=0}^{\infty} [R_k(g)]^{n_k} 
		= \prod_{k=0}^{\infty} \omega_{p_k}^{g_k n_k}, \qquad n \in \mathbb{N}, 
	\end{equation}
	where $n_k$ are the digits of $n$ in the mixed radix number system with bases $p_k$. Specifically, $n_k \in \{0, \ldots, p_k-1\}$ are the coefficients in the expansion 
	$n = \sum_{k=0}^\infty n_k P_k$ (with only finitely many $n_k$ nonzero).
\end{definition}

Note that, up to a countable set, the group $\mathbb{G}_{\mathcal{P}}$ is mapped onto the interval $[0,1]$ via
$$\phi: \mathbb{G}_{\mathcal{P}} \to [0,1], \qquad \phi(g) = \sum_{k=0}^{\infty} \frac{g_k}{P_{k+1}}.$$
If $p_k = 2$ for all $k$, the system is called the Walsh system. The dyadic Rademacher system is denoted by $(r_i)_{i\in \mathbb{N}}$. 

The Vilenkin system is said to be \textit{bounded} if $\sup_{k \in \mathbb{N}_0} p_k \le M$ for some $M$, and \textit{unbounded} if $\sup_{k \in \mathbb{N}_0} p_k = \infty$.

We denote by $B_N$ the set of the first $P_N$ Vilenkin functions:
\begin{equation} \label{B_N}
	B_N := \{V_n : 0 \le n < P_N\}.
\end{equation} 

In the literature, the functions defined above are also referred to as Vilenkin--Crestenson functions \cite{Malozemov-Masharskii-2001}.

	\section{ $d$-dissociated subsystems of the character. Definition and Examples.} \label{section-1}
	
	We introduce the notion of $d$-dissociativity. For the case $d=2$, see \cite[(37.12)]{HewRoss}.
	
	\begin{definition} \label{dissociate}
		Let $\Gamma$ be the character group of a compact abelian group $G$. A finite subset $\{\gamma_1, \ldots, \gamma_m\} \subset \Gamma$ (with all $\gamma_i$ distinct) is called $d$-dissociated if it does not contain the trivial character and if the relation
		\[
		\gamma_1^{k_1} \gamma_{2}^{k_2} \dots \gamma_{m}^{k_m} \equiv 1,
		\]
		where each $k_j$ belongs to $\{-d, -(d-1), \ldots, 0, \ldots, d-1, d\}$, implies that
		\[
		\gamma_1^{k_1} \equiv \gamma_{2}^{k_2} \equiv \ldots \equiv \gamma_{m}^{k_m} \equiv 1.
		\]
		An infinite subset of $\Gamma$ is said to be $d$-dissociated if every one of its finite subsets is $d$-dissociative.
	\end{definition}
	
	Let us give some examples.
	\begin{exmpl}
		Let $G = \mathbb{T}$ be the one-dimensional torus. In this case, the system of characters is the trigonometric system $\{\exp(nix)\}_{n\in \mathbb{Z}}$. Consider a Hadamard lacunary sequence $\{n_k\}_{k=1}^{\infty}$, i.e., $\frac{n_{k+1}}{n_k} \ge q$. Let $q \ge d+1$; then the subsystem $\{\exp(\pm n_kix)\}_{k=1}^{\infty}$ is $d$-dissociated. Let us prove this fact. Assume that the system is not $d$-dissociated. Then there exist functions $e^{in_{\alpha_1} x}, \ldots, e^{in_{\alpha_m} x}$, with $n_{\alpha_1} < \ldots < n_{\alpha_m}$, and integers $a_1, \ldots, a_m \in \{-d, \ldots, 0, \ldots, d\} \setminus \{0\}$ such that
		\[
		\exp\bigl([a_1 n_{\alpha_1} + \ldots + a_m n_{\alpha_m}]\, i x\bigr) \equiv 1,
		\]
		but $e^{a_j n_{\alpha_j} i x} \not\equiv 1$ for all $j \in \{1, \ldots, m\}$.
		For $q \geq d+1$ we have
		\[
		\biggl|\sum_{i=1}^{m-1} a_i n_{\alpha_i}\biggr| 
		\leq |n_{\alpha_m}| \cdot d \sum_{i=1}^{m-1} \frac{1}{q^i} 
		< \frac{|n_{\alpha_m}| \cdot d}{q-1} \le |n_{\alpha_m}|.
		\]
		The last relation contradicts the fact that $a_1 n_{\alpha_1} + \ldots + a_m n_{\alpha_m} \equiv 0.$ Thus, the system is $d$-dissociated.
	\end{exmpl}
	
	\begin{exmpl} \label{prikol} Let the sequence $(p_k)_{k\in\mathbb{N}}$ be such that $p_k= p\ge 2$ for all $k$. Let us consider a finite set of Vilenkin functions.
		\begin{equation} 
			\label{finitesubsetV^d_p}
			V_{n^1}, \; V_{n^2}, \; \ldots, \; V_{n^m}, 
			\qquad 
			n^1 < n^2 < \; \ldots, \; < n^m. 
		\end{equation}
		For each $n^i$, we write its $p$-adic expansion $ 
		n^i = n_{i1} p^{k_{i1}} + n_{i2} p^{k_{i2}} + \ldots + n_{id(i)} p^{k_{id(i)}}, \;
		n_{ij} \in \{1, \ldots, p\}$.
		We get a set of sets
		\[
		\{ k_{11}, k_{12}, \ldots, k_{1d(1)}\}, 
		\quad 
		\{ k_{21}, k_{22}, \ldots, k_{2d(2)}\}, 
		\ldots, 
		\quad 
		\{ k_{m1}, k_{m2}, \ldots, k_{md(m)} \} 
		\]   of non-negative integers. Let these sets be ordered so that each successive set contains a number that is not in all the preceding ones. Then the corresponding system \eqref{finitesubsetV^d_p} is $d$-dissociated for every $d$.
	\end{exmpl}
	
	We note that, according to \cite{Kazakova-2}, if $p$ is prime, then the condition from Example \ref{prikol} is sufficient for the independence of the functions $V_{n^1}$, $V_{n^2}$, \ldots, $V_{n^m}$.
	
	A particular case of Example \ref{prikol} is the following. Let $(j_k)_{k=1}^{\infty}$ be a fixed sequence of numbers from the set $\{1, \ldots, p-1\}$. Then the system of functions $(R_k^{j_k})_{k=1}^{\infty}$ is a $d$-dissociated system for every $d$ and a system of independent functions for any base $p$.
	\section{Constructions of measures such as Riesz products} \label{section-2}
Let $\{\gamma_{i}(x)\}_{i=1}^{\infty} \subset \Gamma$ be an arbitrary $d$-dissociated subsystem of characters of a compact abelian group $G$, and let $m_i$ be given by \eqref{m_k}. Consider a sequence of complex numbers $(a_{ik})$, where $i \in \mathbb{N}$ and $1 \le k \le d$, such that $|a_{ik}| \le 1.$
	
Then we can define the following analogue of the Riesz product:
	\begin{equation} \label{yyy}
		\rho= weak^*\lim_{n \rightarrow\infty} \rho_n, \qquad  \rho_n(x)= \prod_{i=0}^{n}\left(1 + \frac{1}{2d}\left[\sideset{}{'}{\sum}_{k} a_{ik} \gamma_i^k(x) +  \sideset{}{'}{\sum}_{k} \overline{ a_{ik}\gamma_i^{k}(x)}\right]\right),
	\end{equation}
where the summation $\sum_{k}^{\prime}$ runs over $k \in \{1, \dots, d_i\}$ with $d_i := \min(d, m_i-1)$, excluding indices $k$ for which there exists $k^{\prime} < k$ such that
	\[
	\overline{\gamma_i^{k}} = \gamma_i^{k^{\prime}}.
	\]
We compute the Fourier coefficients with respect to the system $\Gamma$:
	\begin{equation} \label{rho-main}
		\begin{split}
			&\widehat{\rho}(\gamma_{k_1}^{\alpha_{1}} \cdot \ldots \cdot \gamma_{k_j}^{\alpha_{j}}) = \lim_{n\rightarrow \infty} 	\widehat{\rho}_n(\gamma_{k_1}^{\alpha_{1}} \cdot \ldots \cdot \gamma_{k_j}^{\alpha_{j}}) = \frac{a_{k_1\alpha_1}'\ldots a_{k_j\alpha_j}'}{(2d)^j}, \\
			& \text{where} \ k_1 < \ldots < k_j, \ \alpha_{j} \in \{1, \ldots, d_{k_j}\}, \ a_{k_i\alpha_i}' := 
			\begin{cases} 
				\overline{a_{k_i\alpha}}, & \text{if } \exists\, \alpha < \alpha_i : \gamma_{k_i}^{\alpha} = \overline{\gamma_{k_i}^{\alpha_i}}, \\ 
				a_{k_i\alpha_i}, & \text{otherwise}.
			\end{cases} 
		\end{split}
	\end{equation}
	\begin{lemma} The formula \eqref{yyy} correctly defines a measure on the group $G$, with the variation of this measure being equal to 1.
	\end{lemma}
	\begin{proof}
		We obtain
		\begin{equation} \label{corect}
			1 + \frac{1}{2d}\left[\sideset{}{'}{\sum}_{k} a_{ik} \gamma_i^k(x) +  \sideset{}{'}{\sum}_{k} \overline{ a_{ik}\gamma_i^{k}(x)}\right] = 1 + \frac{1}{d} \sideset{}{'}{\sum}_{k}\mathrm{Re}(a_{ik}\gamma_i^k) \ge 0,   
		\end{equation}
		since $|a_{ik}| \le 1$ and $|\gamma_{i}^k(x)|=1$,  because $\{\gamma_{i}(x)\}_{i=1}^{\infty}$ is the characters of a compact Abelian group.
		
		This yields 
		\begin{equation*} 
			\|\rho\|_{L_1} = \int^1_0 |\rho_n(x)| d \mu \overset{(\star)}{=} \int_0^1 \rho_n(x)d\mu  \overset{(\star \star)}{=} 1, 
		\end{equation*}
		
where $(\star)$ follows from \eqref{corect}, and $(\star\star)$ is justified by the orthogonality of characters and the $d$-dissociativity of the system $\{\gamma_{i}(x)\}_{i=1}^{\infty}$. 

Since $\|\rho_n\|_{L_1} = 1$, we may regard $\rho_n$ as linear functionals on $C[0,1]$. By the Banach--Alaoglu theorem, we can extract a weak$^*$-convergent subsequence, whose limit is the measure $\rho$ given in \eqref{yyy}.    
	\end{proof}

	\begin{lemma} \label{lemma-ppp} Let some $s\le d$ be fixed. Then there exists a measure $\nu_s$ such that for its Fourier coefficients according to the system of characters $\Gamma$, the following is true:
		\[
		\begin{cases}
			\widehat{\nu}_s(\gamma_{k_1}^{\alpha_{1}} \cdot \ldots \cdot \gamma_{k_j}^{\alpha_{j}})  =1, \quad j=s, \\
			\widehat{\nu}_s(\gamma_{k_1}^{\alpha_{1}} \cdot \ldots \cdot \gamma_{k_j}^{\alpha_{j}}) =0, \quad j \le d, j \neq s,
		\end{cases}
		\]
		where $k_1 < \ldots < k_j$ and $\alpha_{j} \in \{1, \ldots, d_{k_j}-1, d_{k_j}\}$.
	\end{lemma}
	\begin{proof}
Consider the measure $\rho$ defined by \eqref{yyy} with $a_{ik} \equiv 1$. Let
		\[
		\nu_s = c_0\delta_0 + c_1\rho + \ldots + c_{d}\underbrace{\rho * \ldots * \rho}_{d}, \qquad \text{where $\delta_0$ --- the Dirac measure at 0}.
		\]
		Then, using the formula \eqref{rho-main}, for the Fourier coefficients of this measure, we obtain:
		\[
		\widehat{\nu}_s(\gamma_{k_1}^{\alpha_{1}} \cdot \ldots \cdot \gamma_{k_j}^{\alpha_{j}}) = c_0 + \sum_{l=1}^{d}c_l (\widehat{\rho}(\gamma_{k_1}^{\alpha_{1}} \cdot \ldots \cdot \gamma_{k_j}^{\alpha_{j}}))^l = \sum_{l=0}^{d}c_l (2d)^{-jl}.
		\]
		We define the coefficients $\{c_i\}_{i=0}^{d}$ as the coefficients of the polynomial of degree $d$ that vanishes at $0$ and at $(2d)^{-j}$ for all $j \le d$ with $j \ne s$, and takes the value $1$ at $(2d)^{-s}$.
		
		For a variation of the measure $\nu_s$, we get
		\begin{equation} \label{2804}
			||\nu_s|| \le \sum_{i=0}^{d} |c_i| =:C_s.		    
		\end{equation}
	\end{proof}
	\begin{lemma} \label{muy}
Let $R_i$ denote the generalized Rademacher functions with base $2d+1$. Then, for each fixed $y \in [0,1]$, there exists a measure $\rho_y$ such that for any fixed indices $k_1 < \cdots < k_s$ and $\alpha_j \in \{1, \ldots, d_{k_j}\}$, the corresponding coefficients of $\rho_y$ are given by the following formula:
		\begin{equation}  \label{muycoef}
			\widehat{\rho_y}( \gamma_{k_1}^{\alpha_1} \cdot \ldots \cdot \gamma_{k_s}^{\alpha_s}) = \frac{ R_{k_1}^{\alpha_1^{\prime}}(y) \cdot \ldots \cdot R_{k_s}^{\alpha_s^{\prime}}(y)}{(2d)^{s}},   \end{equation}
		\begin{equation} \label{alpha'}
			\text{where} \quad \alpha'_i = \alpha'_i(\alpha_i, k_i) := 
			\begin{cases} 
				2d+1-\alpha, & \text{if } \exists\, \alpha < \alpha_i : \gamma_{k_i}^{\alpha} = \overline{\gamma_{k_i}^{\alpha_i}}, \\ 
				\alpha_i, & \text{otherwise}.
			\end{cases} 
		\end{equation}
	\end{lemma}	
	\begin{proof}
The measure $\rho_y$ is given by \eqref{yyy} with $a_{ik} = R_i^k(y)$, and has total variation $\|\rho_y\| = 1$. Moreover, \eqref{muycoef} and \eqref{alpha'} follow directly from \eqref{rho-main}.
	\end{proof}
	\begin{remark} \label{remark-1}
		Note that the base $2d+1$ for Rademacher functions is chosen due to the fact that in this case, according to the formula \eqref{alpha'}, if $\gamma_{k_1}^{\alpha_1}\cdot\ldots\cdot \gamma_{k_s}^{\alpha_s}$ and $\gamma_{k_1}^{\beta_1} \cdot\ldots \cdot \gamma_{k_s}^{\beta_s}$ various functions, then $R_{k_1}^{\alpha_1'} \cdot\ldots\cdot R_{k_s}^{\alpha_s'}$ and $R_{k_1}^{\beta_1'} \cdot\ldots \cdot R_{k_s}^{\beta_s'}$ --- various functions.
	\end{remark}

	Similarly, we can fix $x\in G$ and consider a measure of the Riesz product type corresponding to the generalized Rademacher system with base $2d + 1$:
	
	\begin{equation} \label{rho_x}
		\rho_x =  weak^*\lim_{n \rightarrow\infty} \prod_{i=0}^{n} \left( 1 + \frac{1}{2d}\left( \sideset{}{'}{\sum}_{k} \gamma_i^{-k}(x) R_i^{k}  + \sideset{}{'}{\sum}_{k} \overline{\gamma_i^{-k}(x) R_i^{k}}\right) \right),
	\end{equation}
Then, for fixed $k_1 < \cdots < k_s$ and $\alpha_j \in \{1, \dots, d_{k_j}\}$, the corresponding coefficients of this measure are given by the formulas:
	\begin{equation} \label{rhox} 
		\begin{split}
			& \quad \widehat{\rho_x}(R_{k_1}^{\alpha_1^{\prime}}(y) \cdot \ldots \cdot R_{k_s}^{\alpha_s^{\prime}}(y)) = \frac{\gamma_{k_1}^{-\alpha_1} \cdot \ldots \cdot \gamma_{k_s}^{-\alpha_s}}{(2d)^{s}}, \\
			& \quad \widehat{\rho_x}(R_{k_1}^{-\alpha_1^{\prime}}(y) \cdot \ldots \cdot R_{k_s}^{-\alpha_s^{\prime}}(y)) = \overline{\widehat{\rho_x}(R_{k_1}^{\alpha_1^{\prime}}(y) \cdot \ldots \cdot R_{k_s}^{\alpha_s^{\prime}}(y))} = \frac{\gamma_{k_1}^{\alpha_1} \cdot \ldots \cdot \gamma_{k_s}^{\alpha_s}}{(2d)^{s}}, 
		\end{split} 
	\end{equation}
	where $\alpha_j^{\prime}$ are defined in \eqref{alpha'}.
	\section{On Lacunarity and Sidon-Type Properties \newline
		 for Polynomial $d$-Chaoses} \label{section-3}
	\begin{theorem} \label{th-0}
Let $\{\gamma_{k}\}_{k=0}^{\infty}$ be a $d$-dissociated system of characters of a compact abelian group, and let $m_k$ be given by \eqref{m_k}. Then, for every $q \ge 1$, including $q = \infty$, and for every polynomial of the form
		\[
		Q^{(s)} := \sum_{0\le k_1 < \ldots < k_s \le N} \; \sum_{\substack{\alpha_1+\ldots+\alpha_s = d \\ 1 \le \alpha_i < m_{k_i}}} C_{k_1, \ldots, k_s}^{\alpha_1, \ldots, \alpha_s} \gamma_{k_1}^{\alpha_1}(x) \cdot \ldots \cdot \gamma_{k_s}^{\alpha_s}(x), \qquad s \le d
		\]
		there exist polynomials 
		\begin{equation} \label{Q-s_R}
			\begin{split}
				& Q^{(s)}_{R, -}(y) := \sum_{0\le k_1 < \ldots < k_s \le N} \; \sum_{\substack{\alpha_1+\ldots+\alpha_s = d \\ 1 \le \alpha_i < m_{k_i}}} C_{k_1, \ldots, k_s}^{\alpha_1, \ldots, \alpha_s} R_{k_1}^{-\alpha^{\prime}_1}(y) \cdot \ldots \cdot R_{k_s}^{-\alpha^{\prime}_s}(y), \\
				& Q^{(s)}_{R, +}(y) := \sum_{0\le k_1 < \ldots < k_s \le N} \; \sum_{\substack{\alpha_1+\ldots+\alpha_s = d \\ 1 \le \alpha_i < m_{k_i}}} C_{k_1, \ldots, k_s}^{\alpha_1, \ldots, \alpha_s} R_{k_1}^{\alpha^{\prime}_1}(y) \cdot \ldots \cdot R_{k_s}^{\alpha^{\prime}_s}(y),
			\end{split}
		\end{equation}
		such that:
		\begin{equation} \label{th-0-main}
			\frac{1}{(2d)^{2d}} \|Q^{(s)}_{R, +}\|_q	\le \|Q^{(s)}\|_q \le (2d)^{2d} \|Q^{(s)}_{R, -}\|_q,
		\end{equation}
		where $\{R_{k}\}_{k=0}^{\infty}$ are Rademacher functions with base $2d+1$ and $\alpha_j'$ are defined by the formula \eqref{alpha'}.
	\end{theorem}
	\begin{proof}
		Consider the following polynomials from $x\in G$ and $y\in [0,1)$:
		\[
		\begin{split}
			& Q^{(s)}_{-}(x,y) = \sum_{0\le k_1 < \ldots < k_s \le N} \; \sum_{\substack{\alpha_1+\ldots+\alpha_s = d \\ 1 \le \alpha_i < m_{k_i}}} C_{k_1, \ldots, k_s}^{\alpha_1, \ldots, \alpha_s} R_{k_1}^{-\alpha^{\prime}_1}(y) \cdot \ldots \cdot R_{k_s}^{-\alpha^{\prime}_s}(y) \gamma_{k_1}^{\alpha_1}(x) \cdot \ldots \cdot \gamma_{k_s}^{\alpha_s}(x),\\
			& Q^{(s)}_{+}(x,y) = \sum_{0\le k_1 < \ldots < k_s \le N} \; \sum_{\substack{\alpha_1+\ldots+\alpha_s = d \\ 1 \le \alpha_i < m_{k_i}}} C_{k_1, \ldots, k_s}^{\alpha_1, \ldots, \alpha_s} R_{k_1}^{\alpha^{\prime}_1}(y) \cdot \ldots \cdot R_{k_s}^{\alpha^{\prime}_s}(y) \gamma_{k_1}^{\alpha_1}(x) \cdot \ldots \cdot \gamma_{k_s}^{\alpha_s}(x),
		\end{split}
		\]
where the tuple $(\alpha_1', \ldots, \alpha_s')$ is constructed from $(\alpha_1, \ldots, \alpha_s)$ and $(k_1, \ldots, k_s)$ via formula \eqref{alpha'}. For a fixed $y \in [0,1)$, consider the measure $\rho_y$ defined in Lemma \ref{muy}. Applying \eqref{muycoef}, we obtain: 
		\[
		Q^{(s)}_{-}(x,y) * \rho_y(x) = \int_{G} Q^{(s)}_{-} (x\dot-z, y) d \rho_y(z) 
		\]
		\[
		=\sum_{0\le k_1 < \ldots < k_s \le N} \; \sum_{\substack{\alpha_1+\ldots+\alpha_s = d \\ 1 \le \alpha_i < m_{k_i}}} C_{k_1, \ldots, k_s}^{\alpha_1, \ldots, \alpha_s} R_{k_1}^{-\alpha^{\prime}_1}(y) \cdot \ldots \cdot R_{k_s}^{-\alpha^{\prime}_s}(y) \cdot  \gamma_{k_1}^{\alpha_1}(x) \cdot \ldots \cdot \gamma_{k_s}^{\alpha_s}(x) \cdot \widehat{\rho_y}(\gamma_{k_1}^{\alpha_1} \cdot \ldots \cdot \gamma_{k_s}^{\alpha_s})
		\]
		\[
		= \frac{1}{(2d)^s} \sum_{0\le k_1 < \ldots < k_s \le N} \; \sum_{\substack{\alpha_1+\ldots+\alpha_s = d \\ 1 \le \alpha_i < m_{k_i}}} C_{k_1, \ldots, k_s}^{\alpha_1, \ldots, \alpha_s} \gamma_{k_1}^{\alpha_1}(x) \cdot \ldots \cdot \gamma_{k_s}^{\alpha_s}(x) = \frac{Q^{(s)}(x)}{(2d)^s}.
		\]
		Then, by Young's inequality, 
		\begin{equation} \label{Qy}
			\|Q^{(s)}(x)\|_{L_q(x)} = (2d)^{s}\|Q^{(s)}_{-}(x,y)* \rho_y\|_{L_q(x)} \le (2d)^{d}\|Q^{(s)}_{-}(x,y)\|_{L_q(x)}. 
		\end{equation}
		Taking the $L_q$-norm with respect to $y$ of both sides of \eqref{Qy}, we obtain:
		\begin{equation} \label{Qy2}
			\|Q^{(s)} (x)\|_{L_q(x)} \le (2d)^d \|Q^{(s)}_{-}(x,y)\|_{L_q(x)\times L_q(y)}. 
		\end{equation}
		
	Consider the polynomial $Q^{(s)}_{R, -}$, defined by the formula \eqref{Q-s_R}, and consider for a fixed $x\in G$ the measure $\rho_x$, defined by the formula \eqref{rho_x}. Calculate $Q^{(s)}_{R, -}(y)*\rho_x$, taking into account \eqref{rhox},
		\[
		Q^{(s)}_{R, -}(y) * \rho_x(y) = \int_0^1 Q^{(s)}_{R, -} (y\dot-z) d \rho_x(z) 
		\]
		\[
		=\sum_{0\le k_1 < \ldots < k_s \le N} \; \sum_{\substack{\alpha_1+\ldots+\alpha_s = d \\ 1 \le \alpha_i < m_{k_i}}} C_{k_1, \ldots, k_s}^{\alpha_1, \ldots, \alpha_s} R_{k_1}^{-\alpha^{\prime}_1}(y) \cdot \ldots \cdot R_{k_s}^{-\alpha^{\prime}_s}(y) \cdot \widehat{\rho_x}(R_{k_1}^{-\alpha^{\prime}_1}(y) \cdot \ldots \cdot R_{k_s}^{-\alpha^{\prime}_s}(y))
		\]
		\[
		= \frac{1}{(2d)^s} \sum_{0\le k_1 < \ldots < k_s \le N} \; \sum_{\substack{\alpha_1+\ldots+\alpha_s = d \\ 1 \le \alpha_i < m_{k_i}}} C_{k_1, \ldots, k_s}^{\alpha_1, \ldots, \alpha_s} R_{k_1}^{-\alpha^{\prime}_1}(y) \cdot \ldots \cdot R_{k_s}^{-\alpha^{\prime}_s}(y) \gamma_{k_1}^{\alpha_1}(x) \cdot \ldots \cdot \gamma_{k_s}^{\alpha_s}(x) = \frac{Q^{(s)}_{-}(x,y)}{(2d)^s}.
		\]
		Then, by Young's inequality, 
		\begin{equation} \label{Qx}
			\|Q^{(s)}_{-}(x,y)\|_{L_q(y)} = (2d)^{s}\|Q^{(s)}_{R,-}(y)* \rho_x\|_{L_q(y)} \le (2d)^{d}\|Q^{(s)}_{R,-}(y)\|_{L_q(y)}. 
		\end{equation}
		Taking the $L_q$-norm with respect to $x$ of both sides of \eqref{Qx}, we obtain:
		\begin{equation} \label{Qx2}
			\|Q^{(s)}_{-}(x,y)\|_{L_q(x)\times L_q(y)} \le (2d)^{d}\|Q^{(s)}_{R,-}(y)\|_{L_q(y)}.
		\end{equation}
		From \eqref{Qx2} and \eqref{Qy2} we get:
		\[
		\|Q^{(s)} (x)\|_{L_q(x)} \le (2d)^{2d} \|Q^{(s)}_{R,-}(y)\|_{L_q(y)}.
		\]
		Thus, the right inequality in \eqref{th-0-main} is proved.
		
        The left-hand inequality is proved similarly. For a fixed $y \in [0,1]$, using \eqref{muycoef}, we obtain:  
		\[
		Q^{(s)} * \rho_y = \int_{G} Q^{(s)} (x\dot-z) d\rho_y(z) = \frac{Q^{(s)}_{+}(x,y)}{(2d)^s}.
		\]
Applying Young's inequality and taking the $L_q$-norm with respect to $y$ (as in \eqref{Qy} and \eqref{Qy2}), we obtain:
		\begin{equation} \label{Qy2-l}
			\frac{1}{(2d)^d}\|Q^{(s)}_{+}(x,y)\|_{L_q(x)\times L_q(y)} \le \|Q^{(s)} (x)\|_{L_q(x)}.
		\end{equation}
		Taking into account \eqref{rhox} and \eqref{Q-s_R}, we obtain:
		\[
		Q^{(s)}_{+}(x,y) * \rho_x = \int_0^1 Q^{(s)}_{+} (x, y \dot-z)d\rho_x(z) = \frac{Q^{(s)}_{R, +}(y)}{(2d)^s}.
		\]
		Applying Young's inequality and taking the $L_q$-norm with respect to $x$ (as in \eqref{Qx} and \eqref{Qx2}), we obtain:
		\begin{equation} \label{Qx2-l}
			\frac{1}{(2d)^{d}} \|Q^{(s)}_{R,+}(y)\|_{L_q(y)}  \le  \|Q^{(s)}_{+}(x,y)\|_{L_q(x)\times L_q(y)}.
		\end{equation}
		From \eqref{Qy2-l} and \eqref{Qx2-l} we get the left inequality in \eqref{th-0-main}:
		\[
		\frac{1}{(2d)^{2d}} \|Q^{(s)}_{R,+}(y)\|_{L_q(y)} \le \|Q^{(s)} (x)\|_{L_q(x)}.
		\]	 
		This completes the proof.
	\end{proof}

	\begin{theorem} \label{THEOREM-1} 
The polynomial chaos $B^d_{\gamma}$, constructed according to the formula \eqref{Bphid} for some $d$-dissociated subsystem $\{\gamma_{i}(x)\}_{i=1}^{\infty}$ of the characters of a compact abelian group, is a system $q$-lacunarities for all $q>2$. That is, for every polynomial with respect to the system $B^d_{\gamma}$:
		\begin{equation}\label{chaosQ}
			Q := \sideset{}{'}{\sum} A_{k_1, \ldots, k_d} \gamma_{k_1}(x) \cdot \ldots \cdot \gamma_{k_d}(x), \quad A_{k_1, \ldots, k_d} \in \mathbb{C},
		\end{equation} 
 where the primed sum $\sum'$ runs over non-decreasing sequences $0 \leq k_1 \leq \cdots \leq k_d \leq N$ such that each index $j \in \{0, 1, \ldots, N\}$ appears in $(k_1, \ldots, k_d)$ fewer than $m_j$ times (with $m_j$ given by \eqref{m_k}), the $L_2$-$L_q$ Khintchine inequality holds:
		\begin{equation} \label{THEOREM-1-eq}
			\|Q\|_{q} \le C \left(\sideset{}{'}{\sum} A_{k_1, \ldots, k_d}^2\right)^{\frac{1}{2}}, \quad C>0 \; \text{does not depend on the polynomial $Q$}.
		\end{equation}
	\end{theorem}
	\begin{proof} We write the polynomial $Q$ from \eqref{chaosQ} in the form
		\begin{equation} \label{C-A}
			\sum_{s=1}^{d} \; \sum_{0\le k_1 < \ldots < k_s \le N} \; \sum_{\substack{\alpha_1+\ldots+\alpha_s = d \\ 1 \le \alpha_i < m_{k_i}}} C_{k_1, \ldots, k_s}^{\alpha_1, \ldots, \alpha_s} \gamma_{k_1}^{\alpha_1}(x) \cdot \ldots \cdot \gamma_{k_s}^{\alpha_s}(x) =: \sum_{s=1}^{d} Q^{(s)},  
		\end{equation} 
		\begin{equation*} 
			\text{where}\; C_{k_1, \dots, k_s}^{\alpha_1, \dots, \alpha_s} = A_{\underbrace{k_1, \dots, k_1}_{\alpha_1}, \underbrace{k_2, \dots, k_2}_{\alpha_2}, \dots, \underbrace{k_s, \dots, k_s}_{\alpha_s}}.
		\end{equation*}
		It is enough to prove the theorem for the polynomial $Q^{(s)}$, since by the Minkowski and Cauchy-Bunyakovsky inequalities, we have
		\begin{equation} \label{MK}
			\|Q\|^2_q \le \left(\sum_{s=1}^d \|Q^{(s)}\|_q\right)^2 \le d \sum_{s=1}^d \|Q^{(s)}\|_q^2. 
		\end{equation}
By Theorem~\ref{th-0}, for the polynomial $Q^{(s)}$ there exists a polynomial $Q^{(s)}_{R,-}$, given by \eqref{Q-s_R}, such that: 
		\begin{equation} \label{q}
			\|Q^{(s)}\|_q \le (2d)^{2d}\|Q^{(s)}_{R, -}\|_q.
		\end{equation}
	By Theorem C, for the polynomial $Q^{(s)}_{R, -}$, we have:
		\begin{equation} \label{1-C}
        \begin{split}
            \|Q^{(s)}_{R, -}\|_q \le \kappa \biggl(\sum_{0\le k_1 < \ldots < k_s \le N} \; \sum_{\substack{\alpha_1+\ldots+\alpha_s = d \\ 1 \le \alpha_i < m_{k_i}}} \left(C_{k_1, \ldots, k_s}^{\alpha_1, \ldots, \alpha_s}\right)^2\biggl)^{\frac{1}{2}}, & \\
            \text{$\kappa>0$ does not depend on $C_{k_1, \ldots, k_s}^{\alpha_1, \ldots, \alpha_s}$},&
        \end{split}
		\end{equation}
		where the $\ell_2$-norm of the coefficients of the polynomial $Q^{(s)}_{R, -}$ is given by \eqref{1-C} and coincides with the $\ell_2$-norm of the coefficients of $Q^{(s)}$ by Remark \ref{remark-1}. Combining \eqref{MK}, \eqref{q}, and \eqref{1-C}, we obtain:
		\[
		\|Q\|_q \le C \biggl(\sum_{s=1}^{d}\sum_{0\le k_1 < \ldots < k_s \le N} \; \sum_{\substack{\alpha_1+\ldots+\alpha_s = d \\ 1 \le \alpha_i < m_{k_i}}} \left(C_{k_1, \ldots, k_s}^{\alpha_1, \ldots, \alpha_s}\right)^2\biggl)^{\frac{1}{2}}, \qquad C := \sqrt{d \kappa} (2d)^{2d},
		\]
		which is exactly formula \eqref{THEOREM-1-eq}. This completes the proof.
	\end{proof}	
	\begin{theorem} \label{THEOREM-2} 
		The polynomial chaos $B^d_{\gamma}$, constructed according to the formula \eqref{Bphid} for some $d$-dissociated subsystem $\{\gamma_{i}(x)\}_{i=1}^{\infty}$ of the characters of a compact abelian group, is a $\frac{2d}{d+1}$-Sidon system, i.e. for every polynomial of the form \eqref{chaosQ} with respect to the system $B^d_{\gamma}$, the $\frac{2d}{d+1}$-Sidon inequality holds:
		\begin{equation} \label{THEOREM-2-eq}
			\left(\sideset{}{'}{\sum} A_{k_1, \ldots, k_d}^{\frac{2d}{d+1}}\right)^{\frac{d+1}{2d}} \le C\|Q\|_{\infty}, \quad C>0 \; \text{does not depend on the polynomial $Q$},
		\end{equation}
		where $\Sigma'$ is defined in the formula \eqref{chaosQ}.
	\end{theorem}
	\begin{proof}
		We show that it suffices to prove the statement for a polynomial of the form $Q^{(s)}$ defined in \eqref{C-A}. Consider the measure $\nu_s$ defined in Lemma \ref{lemma-ppp}. We obtain the following chain of equalities:
		\begin{equation*}
			\begin{split}
				& Q * \nu_s = \int_0^1 Q(x \dot- y) \, d\nu_s(y) = \\
				&= \sum_{s=1}^{d} \; \sum_{0\le k_1 < \ldots < k_s \le N} \; \sum_{\substack{\alpha_1+\ldots+\alpha_s = d \\ 1 \le \alpha_i < m_{k_i}}} C_{k_1, \ldots, k_s}^{\alpha_1, \ldots, \alpha_s} \gamma_{k_1}^{\alpha_1}(x) \cdot \ldots \cdot \gamma_{k_s}^{\alpha_s}(x)\cdot \widehat{\nu_s}(\gamma_{k_1}^{\alpha_1} \cdot \ldots \cdot \gamma_{k_s}^{\alpha_s}) = Q^{(s)}.
			\end{split}
		\end{equation*}
		
		Then, using Young's inequality for convolution and formula \eqref{2804}, we obtain 
		\begin{equation} \label{28041900}
			\|Q^{(s)} \|_{\infty} =\|Q * \nu_s \|_{\infty}  \le \|\nu_s\|\|Q \|_{\infty}  \le C_s\|Q \|_{\infty}.
		\end{equation}
		Let's raise both sides of inequality \eqref{28041900} to the power $\frac{2d}{d+1}$ and sum over $s$ from 1 to $d$:
		\begin{equation} \label{4-}
			\sum_{s=1}^{d}\|Q^{(s)}\|_{\infty}^{\frac{2d}{d+1}} \le d M^{\frac{2d}{d+1}}\|Q\|_{\infty}^{\frac{2d}{d+1}}, \quad \text{where $M=\max\{C_1, \ldots C_{d}\}$.}
		\end{equation}
		By Theorem~\ref{th-0}, for the polynomial $Q^{(s)}$ there exists a polynomial $Q^{(s)}_{R,+}$, given by \eqref{Q-s_R}, such that:
		\begin{equation} \label{infty}
			\|Q^{(s)}_{R, +}\|_{\infty} \le (2d)^{2d}\|Q^{(s)}\|_{\infty}.
		\end{equation}
		By Theorem D and Remark \ref{remark-1}, for the polynomial $Q^{(s)}_{R, +}$ we have:
		\begin{equation} \label{2-C}
        \begin{split}
            \biggl(\sum_{0\le k_1 < \ldots < k_s \le N}\sum_{\substack{\alpha_1+\ldots+\alpha_s = d \\ 1 \le \alpha_i < m_{k_i}}} \left(C_{k_1, \ldots, k_s}^{\alpha_1, \ldots, \alpha_s}\right)^{\frac{2d}{d+1}}\biggl)^{\frac{d+1}{2d}} \le c\|Q^{(s)}_{R, +}\|_{\infty}, & \\
            \text{$c>0$ does not depend on $C_{k_1, \ldots, k_s}^{\alpha_1, \ldots, \alpha_s}$}. &
        \end{split}
		\end{equation}
		Combining \eqref{4-}, \eqref{infty}, and \eqref{2-C}, we obtain:
		\[
		\biggl(\sum_{s=1}^{d}\sum_{0\le k_1 < \ldots < k_s \le N} \; \sum_{\substack{\alpha_1+\ldots+\alpha_s = d \\ 1 \le \alpha_i < m_{k_i}}} \left(C_{k_1, \ldots, k_s}^{\alpha_1, \ldots, \alpha_s}\right)^{\frac{2d}{d+1}}\biggl)^{\frac{d+1}{2d}}  \le C\|Q\|_{\infty}, \qquad C := (2d)^{2d}cMd^{\frac{d+1}{2d}},
		\]
	which is exactly formula \eqref{THEOREM-2-eq}. This completes the proof.
	\end{proof}
	\section{Necessary Condition for the $\alpha$-Sidon Property of Vilenkin Subsystems} \label{section-6}
	\begin{theorem} \label{theorem-a}
Consider a subsystem of the Vilenkin system with indices in $A$, and let
\[
F_A^N = \bigl\{ x \in [0,1) : V_n(x) = 1 \text{ for all } n \in A \cap B_N \bigr\},
\]
where $B_N$ is defined in \eqref{B_N}. Then, if this subsystem is an $\alpha$-Sidon system for some $1 \le \alpha < 2$, it follows that for every nonempty subset $A' \subset A$,
		\begin{equation} \label{th-a-1}
			|A^{\prime} \cap B_N | \le C \cdot \left(\log\left(\frac{1}{\mu(F_{A^{\prime}}^{N})}\right)\right)^{\frac{\alpha}{2-\alpha}}
		\end{equation}
	\end{theorem}
	\begin{proof}
		For a fixed $t$, consider the polynomial with respect to the Vilenkin system of the following form: $\sum_{n\in A} r_n(t)V_n(x)$, where $r_i(t)$ are Rademacher functions. Assume that the Vilenkin subsystem indexed by $A$ is an $\alpha$-Sidon system. Then, for every subset $A' \subset A$:
		\begin{equation} \label{th-a-2}
			\left(\sum_{n \in A^{\prime} \cap B_N} |r_n(t)|^{\alpha}\right)^{\frac{1}{\alpha}} \le C_1 \left\|\sum_{n \in A^{\prime} \cap B_N} r_n(t)V_n(x)\right\|_{\infty},
		\end{equation}
	where the constant $C_1$ does not depend on the polynomial.
		
		Taking the expectation with respect to $t$ of both sides of \eqref{th-a-2}, we obtain:
		\begin{equation} \label{th-a-3}
			|A^{\prime} \cap B_N |^{\frac{1}{\alpha}} \le C_1 \mathbb{E}_t \left\|\sum_{n \in A^{\prime} \cap B_N} r_n(t)V_n(x)\right\|_{\infty}.
		\end{equation}
		According to Theorem 1.4. from \cite{Marcus-Pisier}, 
		\begin{equation} \label{th-a-4}
			\mathbb{E}_t \left\|\sum_{n \in A^{\prime} \cap B_N} r_n(t)V_n(x)\right\|_{\infty} \le C_2 \left( \sqrt{|A^{\prime} \cap B_N |} + \int_0^1 \frac{\bar{\sigma}(u)}{u \sqrt{\log(\frac{4}{u})}} du \right),
		\end{equation}
		where $\bar{\sigma}(x)$ denote the non-decreasing rearrangement of the function  \[\sigma(x) =  \left(\sum_{n \in A^{\prime} \cap B_N} \bigl|V_n(x)-1\bigl|^2 \right)^{\frac{1}{2}}.\]
		Then we get, \begin{equation} \label{th-a-5} 0 \le \sigma(x) \le 2 \sqrt{|A^{\prime} \cap B_N |}, \qquad \sigma(x) = 0 , \; x \in F_{A^{\prime}}^N. 
		\end{equation}
		Substituting \eqref{th-a-5} into \eqref{th-a-4}, we obtain:
		\begin{equation} \label{th-a-6}
			\mathbb{E}_t \left\|\sum_{n \in A^{\prime} \cap B_N} r_n(t)V_n(x)\right\|_{\infty} \le C_3 \int_{\mu(F_{A^{\prime}}^N)}^1 \frac{\sqrt{|A^{\prime} \cap B_N |}}{u \sqrt{\log(\frac{4}{u})}} du,
		\end{equation}
		Evaluating the integral on the right-hand side of \eqref{th-a-6}, we obtain:
		\begin{equation} \label{th-a-7}
			\mathbb{E}_t \left\|\sum_{n \in A^{\prime} \cap B_N} r_n(t)V_n(x)\right\|_{\infty} \le C_4 \sqrt{|A^{\prime} \cap B_N|} \sqrt{\log \biggl(\frac{1}{\mu(F_{A^{\prime}}^{N})}\biggl)}
		\end{equation}	
	    Combining \eqref{th-a-3} and \eqref{th-a-7}, we obtain:
		\[
		|A^{\prime} \cap B_N |^{\frac{1}{\alpha}}\le C_5 \sqrt{|A^{\prime} \cap B_N|} \sqrt{\log \biggl(\frac{1}{\mu(F_{A^{\prime}}^{N})}\biggl)}
		\]
		Then, 
		\begin{equation} \label{th-a-8}
			|A^{\prime} \cap B_N |^{\frac{1}{\alpha} - \frac{1}{2}}\le C_5 \sqrt{\log \biggl(\frac{1}{\mu(F_{A^{\prime}}^{N})}\biggl)}
		\end{equation}	
		From \eqref{th-a-8} we obtain the estimate \eqref{th-a-1}.
	\end{proof}
	
	The following corollary follows from Theorem~\ref{theorem-a} and has been proved by a different method in \cite[Ch.~VII, \S~11]{Blei}.
	\begin{corollary}
		The polynomial chaos-type subsystem $C^d_R$ of the Vilenkin system with constant $p$, constructed via \eqref{Cphid}, fails to have the $\alpha$-Sidon property for $\alpha < \frac{2d}{d+1}$.
	\end{corollary}
	\begin{proof}
		Let the set $B_N$, as before, denote the first $p^N$ Vilenkin functions, and let $A$ denote the set of indices of the functions in $C^d_R$, viewed as a subsystem of the Vilenkin system. Then $\mu(F_A^N) \ge p^{-N}$ and $|A \cap B_N | = C_N^d p^d \sim \frac{N^d}{d!} p^d$. Substituting into \eqref{th-a-1}, we obtain
		\begin{equation} \label{corollary-a}
			\frac{N^d}{d!} p^d \le C N^{\frac{\alpha}{2-\alpha}}.
		\end{equation}	
		For $\alpha < \frac{2d}{d+1}$, we have $\frac{\alpha}{2-\alpha} < d$, and hence condition \eqref{corollary-a} fails to hold for such $\alpha$. Consequently, by Theorem~\ref{theorem-a}, the corresponding subsystem is not $\alpha$-Sidon.
	\end{proof}
	
\section{Polynomial Chaos-Type Subsystems of Unbounded Vilenkin Systems} \label{section-5}
It readily follows from Theorems \ref{THEOREM-1} and \ref{THEOREM-2} that the polynomial chaos-type subsystems $C^d_R$ of bounded Vilenkin systems, constructed via \eqref{Cphid}, are $q$-lacunary and $\frac{2d}{d+1}$-Sidon. In this section, we consider the case of unbounded Vilenkin systems.

	\begin{theorem} \label{theorem-not-a-sid}
		A subsystem of the type of polynomial chaos $C^d_R$, constructed according to the formula \eqref{Cphid}, an unbounded Vilenkin system is not a $\alpha$-Sidon subsystem for any $1\le\alpha<2$ and for all $d\ge 1$.
	\end{theorem}
	\begin{proof}
		It suffices to show that the subsystem $C_R^1$ of unbounded Vilenkin systems is not a $\alpha$-Sidon system for any $1\le\alpha<2$, then the subsystems $C_R^d$ for all $d\ge1$ will not be systems of $\alpha$-Sidon, since $C_R^{1,d} \subset C_R^d$, where 
		\[
		C_R^{1,d} := \{VC_n \cdot R_0\cdot \ldots \cdot R_{d-1}, \quad VC_n \in \widetilde{C}_R^1 := C_R^1 \setminus \{R_0,\ldots, R_0^{p_0-1}, \ldots, R_{d-1},\ldots, R_{d-1}^{p_{d-1}-1}\}\}.
		\]  
Indeed, if $C_R^{1,d}$ were an $\alpha$-Sidon system for some $1 \le \alpha < 2$, then $\widetilde{C}_R^1$ would also be an $\alpha$-Sidon system for the same $\alpha$, since multiplication by the function $R_0 \cdot \ldots \cdot R_{d-1}$ preserves the $L_q$-norms of polynomials and the $\ell_a$-norms of their coefficients. 
		
Let's prove that the subsystem $C_R^1$ is not a $\alpha$-Sidon system for any $\alpha<2$. Consider the following subsystem for each $k$:
		\begin{equation} \label{theorem-not-a-sid-1}
			A_k := \{R_k, \ldots, R_k^{p_k-1}\}.
		\end{equation}
		Then 
		\begin{equation} \label{theorem-not-a-sid-2}
			F_{A_k}:= \{x\in [0,1)\colon R_k(x)=1\} \; \text{and} \; \mu(F_{A_k}) = \frac{1}{p_k}.
		\end{equation}
By Theorem~\ref{theorem-a}, if $C^1_R$ is an $\alpha$-Sidon system, then there exists a constant $C$ independent of $k$ such that:
		\begin{equation} \label{theorem-not-a-sid-3}
		|A_k|\le C \cdot \left(\log\left(\frac{1}{\mu(F_{A_k})}\right)\right)^{\frac{\alpha}{2-\alpha}}.
	    \end{equation}
		Substituting \eqref{theorem-not-a-sid-1} and \eqref{theorem-not-a-sid-2} into \eqref{theorem-not-a-sid-3}, we obtain
		\begin{equation} \label{theorem-not-a-sid-4}
			p_k\le C \cdot \left(\log p_k\right)^{\frac{\alpha}{2-\alpha}}.
		\end{equation}
		Inequality \eqref{theorem-not-a-sid-4} cannot hold for sequences $(p_k)$ with $\sup_k p_k = \infty$, since $x / (\log x)^{\frac{\alpha}{2-\alpha}} \to \infty$ as $x \to \infty$ for $1 \le \alpha < 2$.Thus, we get a contradiction and the subsystem $C^1_R$, and hence $C^d_R$($d\ge 1$), of the unbounded Vilenkin system is not a $\alpha$-Sidon system for any $\alpha<2$.
	\end{proof}
	
	\begin{theorem} \label{theorem-not-q-lak}
	A subsystem of the type of polynomial chaos $C^d_R$, constructed according to the formula \eqref{Cphid}, an unbounded Vilenkin system is not a subsystem of the $q$-lacunarity system for all $q>2$ and all $d\ge 1$.
	\end{theorem}	
	\begin{proof} Following the same argument as at the beginning of the proof of Theorem~\ref{theorem-not-a-sid}, it suffices to prove the statement only for subsystems of the form $C^1_R$ of unbounded Vilenkin systems. Since $\sup_k p_k = \infty$, there exists a sequence $\{p_{k_s}\}_{s \in \mathbb{N}}$ such that $p_{k_s} \ge 2^s$.
		
Consequently, the series $\sum_{s=1}^{\infty} \frac{1}{p_{k_s}}$ converges. Given any $\varepsilon > 0$, we can choose an index $r$ such that
\begin{equation}
	\label{Eq:Ex-2}
	\sum_{s=r}^{\infty} \frac{1}{p_{k_s}} < \varepsilon.
\end{equation}
Consider the function $g \colon [0,1) \to \mathbb{C}$, 
		\begin{equation}
			\label{Eq:Ex-1} 
			g (x) 
			:= 
			\sum_{s=r}^{\infty} a_s, \qquad 
			a_s := 2^{-s} 
			\frac{ 1 + R_{k_s} (x) + \ldots + R_{k_s}^{p_{k_s}-1} (x)}{p_{k_s}}. 
		\end{equation}
		and we denote:
		\[
		F_s := \{x\in [0,1)\colon R_{k_s}(x) =1\}.
		\]
	Note that $a_s(x) = 2^{-s} \chi_{F_s}(x)$, where $\chi_{F_s}$ is the indicator of the set of $F_s$. 
	Since $|a_s(x)|\le 2^{-s}$ and the series $\sum_{s=r}^{\infty}2^{-s}$ converge,
	the series \eqref{Eq:Ex-1} converges absolutely and uniformly by $[0,1)$.
Therefore, $g(x)$ is correctly defined as the sum of this series.  
In this case, $g(x) = 0$ if $x \in G$, 
		\[ 
		G := 
		[ 0, 1 ) \setminus 
		\Big( 
		\bigcup_{s=r}^{\infty} 
		F_s
		\Big), \qquad \mu (G)  
		\ge 1 - \sum_{s=r}^{\infty} \frac{1}{p_{k_s}} > 1 - \varepsilon,
		\] 
		where the estimate for $\mu(G)$ follows from \eqref{Eq:Ex-2}.
		
		Furthermore, by Gaposhkin's result (see \cite[Ch.~1, \S~2]{Gaposhkin}), if an orthonormal system is such that for every $\varepsilon > 0$ there exists a set $G \subset [0,1)$ with $\mu(G) > 1 - \varepsilon$ and a non-trivial series with respect to this system converging to zero on $G$ (such a system is said to lack the $\varepsilon$-uniqueness property), then this system is not a $q$-lacunarity system for any $q>2$. Since for each $\varepsilon>0$ we have constructed a set and a corresponding series satisfying the conditions of Gaposhkin's theorem, the subsystem $C^1_R$, and hence $C^d_R$($d\ge 1$), of the unbounded Vilenkin system is not a system of $q$-lacunarity.
		This completes the proof. 
	\end{proof}

	\section{Applications} \label{section-4}
	In this section, we make a brief remark concerning the discretization of integral norms in our setting, combining the results of this chapter with known results on discretization. First, we note that the term $q$-lacunary is applied not only to systems (see \eqref{28042039}) but also to subspaces (see \cite[Ch. 9, \S 4]{Kashin-Saakyan-1999}). Observe that if $\{\varphi_n\}$ is a $q$-lacunary system and $L$ is a subspace of $L^2(0, 1)$ consisting of functions of the form
	
	\[
	f(x) \stackrel{L^2}{=} \sum_{n=1}^{\infty} a_n \varphi_n(x), \quad \sum_{n=1}^{\infty} a_n^2 < \infty,
	\]
	
	then it follows directly from definition \eqref{28042039} that any orthonormal basis in $L$ is also a $q$-lacunary system. Therefore, it is often more natural to speak of $q$-lacunary subspaces in $L^2(0, 1)$, i.e., subspaces $L$ such that
	
	\[
	\|f\|_2 \leqslant \|f\|_q \leqslant K \|f\|_2, \quad f \in L \quad (2 < q < \infty).
	\]
	
	By the discretization of the $L_q$-norm on a finite-dimensional space $X_N$, we mean the existence of a set of weights $\{\lambda_i\}_{i=1}^{m}$ and a set of points $\{\xi_i\}_{i=1}^{m}$ such that for every $f \in X_N$,
	
	\[
	C_1\|f\|_q \le \left(\sum_{i=1}^{m}\lambda_i|f(\xi_i)|^{q} \right)^{1/q} \le C_2\|f\|_q.
	\]
	
	It is known that when the $L_q$-norm with $q \in (2,\infty)$ in $N$-dimensional space
	is equivalent to the $L_2$-norm, then for discretization with positive weights of the $L_q$-norm, it is necessary to have at least $O(N^{q / 2})$ points (see, for example, \cite[(D.20)]{Kashin-2022}).
	
	Then, if $B_N \subset B_{\gamma}^d$ is a subsystem of $N$ functions, and $L_{B_N} \subset L_2(0,1)$ is the corresponding $N$-dimensional subspace, then for the discretization of the $L_q$-norm with positive weights of functions from $L_{B_N}$, according to the result stated above and Theorem \ref{THEOREM-1}, $O(N^{q / 2})$ points are required.

\end{document}